\documentclass[hidelinks, 11pt]{amsart}

\usepackage{amsmath}
\usepackage[psamsfonts]{amssymb}
\usepackage{amsfonts}
\usepackage{algorithm2e}
\usepackage{graphicx}
\usepackage{verbatim}
\usepackage{dsfont}
\usepackage{tikz-cd}
\usepackage{mathrsfs}
\usepackage{color}
\usepackage[english]{babel}
\usepackage{fontenc}
\usepackage{indentfirst}
\usepackage{tikz}
\usetikzlibrary{matrix,arrows,decorations.pathmorphing}
\usepackage{algorithm2e}
\usepackage[all]{xy}
\usepackage[T1]{fontenc}
\usepackage{hyperref}
\usepackage{mathtools}
\usepackage{multicol}
\usepackage{enumerate}

\newtheorem{theorem}{Theorem}[section]
\newtheorem{prop}[theorem]{Proposition}
\newtheorem{lemma}[theorem]{Lemma}
\newtheorem{cor}[theorem]{Corollary}

\newtheorem{dfn}[theorem]{Definition}

\numberwithin{equation}{section}

\DeclareMathOperator{\R}{\mathbb{R}}

\def\R{\mathbb{R}}
\def\Z{\mathbb{Z}}

\def\cal R{\mathcal R}

\author{Igor Uljarevi\'c}
\email{igoru@matf.bg.ac.rs}
\address{Faculty of mathematics, University of Belgrade, Studentski trg 16, 11158 Belgrade,
Serbia}

\author{Jun Zhang}
\email{jzhang4518@ustc.edu.cn}
\address{The Institute of Geometry and Physics, University of Science and Technology of China, 96 Jinzhai Road, Hefei Anhui, 230026, China}
 
\title{Contact quasi-states and their applications}

\begin{document}

\maketitle

\begin{abstract}
We introduce the notions of partial contact quasi-state and contact quasi-measure. Using the contact spectral invariant from \cite{DUZ23}, one can construct partial contact quasi-states and contact quasi-measures on each contact manifold fillable by a Liouville domain with non-vanishing and $\Z$-graded symplectic homology. As an application, we present an alternative proof of the contact big fibre theorem from the recent work by Sun-Uljarevi\'c-Varolgunes \cite{SUV25} under a mild topological condition. Our proof follows a more ``classical'' approach developed by Entov-Polterovich \cite{EP06} in the symplectic setting.
\end{abstract}

\section{Introduction}

In analogy to partial symplectic quasi-states and symplectic quasi-measures of Entov-Polterovich \cite{EP06} (see also \cite{PR14}), this note introduces partial contact quasi-states and contact quasi-measures. Concrete examples of partial contact quasi-states and contact quasi-measures can be constructed via contact Hamiltonian Floer theory \cite{MU19} and spectral invariants derived from it \cite{DUZ23}. In particular, we construct a partial contact quasi-state $\zeta_\alpha$ and contact quasi-measure $\tau_\alpha$ on a Liouville-fillable contact manifold $M$ for each contact form $\alpha$ on $M$ provided the Liouville filling has non-vanishing symplectic homology. When restricted to Reeb-invariant functions, the properties of $\zeta_\alpha$ are completely analogous to the properties of partial symplectic quasi-states. 

However, in general situation, these properties often include additional error terms (compare, for instance, the triangle inequalities in Theorem~\ref{thm:spec},~\eqref{item:triangle} and Proposition~\ref{prop:trianglenonstrict}). These error terms are a result of estimating the difference between values of a given function along a Reeb trajectory. One can see the error terms as a manifestation of one typical feature of contact geometry: the constant contact Hamiltonians do not generate constant isotopies, as opposed to constant (symplectic) Hamiltonians. Similarly, the contact quasi-measure $\tau_\alpha$ behaves well on Reeb-invariant subsets, whereas on general subsets less so. 

Nevertheless, partial contact quasi-states and contact quasi-measures are powerful enough to detect non-trivial rigidity phenomena in contact geometry. We demonstrate this by giving an alternative proof of the contact big fibre theorem from \cite{SUV25} under the additional condition that the symplectic homology of the filling is $\Z$-graded\footnote{The condition we actually need is that the unit is not contained in the image of the map ${\rm SH}^\ast(W)\to {\rm SH}_\ast(W)$, so that our spectral invariant is finite. See \cite{C2024} for related discussion.}. 

\begin{theorem} [cf.~Corollary 1.10 in \cite{SUV25}]  \label{thm-bft}Let $(M, \xi)$ be a closed contact manifold that is fillable by a Liouville domain $W$ with non-zero symplectic homology and $c_1(W)=0$. Then, any contact involutive map $F: M \to \R^N$ has a fibre which is not contact displaceable. \end{theorem}

As in \cite{SUV25}, ``contact involutive'' here means that there exists a contact form $\alpha$ on $M$ such that   $F$ is Reeb invariant with respect to $\alpha$ and such that the coordinates of $F,$ seen as contact Hamiltonains (with respect to $\alpha$), generate commuting contact isotopies. Notice that Reeb invariance cannot be removed from the definition of contact involutive map so that Theorem~\ref{thm-bft} still holds. This can be seen by considering an example of a circle filled by a higher-genus surface. 
 
\section{Preliminaries}

\subsection{Notation} Here, let us recall some notations. 

\begin{enumerate}[i)]
\item Throughout the paper, $(W,\lambda)$ denotes a Liouville domain with ${\rm SH}_\ast(W)\not=0$ and $c_1(W)=0,$ $M:=\partial W$ denotes its boundary, and $\alpha:=\left.\lambda\right|_{M}$ the induced contact form. The contact structure $\ker \alpha$ is denoted by $\xi$ and the Reeb flow of $\alpha$ by $\varphi_R^t$.
\item Function $h: [0,1] \times M \to \R$ serves as a contact Hamiltonian function. Often in this paper, $h$ is autonomous (time-independent). The contact isotopy generated by $h$ is denoted by $\varphi_h^t$ for $t \in [0,1]$. 
\item For two contact Hamiltonian functions $g, h: [0,1] \times M \to \R$, denote 
\[ (g \# h)_t = g_t + (\kappa_g^t h_t) \circ (\varphi^t_g)^{-1} \]
 where $\kappa_g^t$ is the conformal factor of $\varphi_g^t$ (with respect to the contact  form $\alpha$), that is, $(\varphi_g^t)^*\alpha = \kappa_g^t \alpha$. Note that $\varphi_{g\#h}^t =\varphi_g^t \varphi_h^t$. 
\item For two contact Hamiltonian functions $g, h: [0,1] \times M \to \R$, denote by $g\bullet h$ the concatenation of the contact Hamiltonians $g$ and $h$ defined by 
    \[ \left(g\bullet h\right)_t:= \left\{ \begin{matrix} 2g_{2t} & \text{for } t\in\left[0,\frac{1}{2}\right]\\ 2h_{2t-1}&\text{for } t\in\left[\frac{1}{2}, 1\right].\end{matrix} \right.\]
Note that $g \bullet h$ generates the contact isotopy 
\[ \varphi^t_{g\bullet h} = \left\{ \begin{matrix} \varphi_g^{2t} & \text{for } t\in\left[0,\frac{1}{2}\right]\\ \varphi_h^{2t-1} \varphi_g^1 &\text{for } t\in\left[\frac{1}{2}, 1\right].\end{matrix} \right.\]
\item Denote by $\widetilde{{\rm Cont}}_0(M)$ the universal cover of the identity component of the contactomorphism group of $(M, \xi)$. Note that in $\widetilde{{\rm Cont}}_0(M)$, we have $[\{\varphi_{g\#h}^t\}] = [\{\varphi^t_{g\bullet h}\}]$.
\item A function $h: [0,1] \times M \to \R$ is called strict with respect to a contact form $\alpha$ if $dh(R_{\alpha}) =0$. A contact isotopy $\{\varphi_t\}_{t \in [0,1]}$ is called strict with respect to a contact form $\alpha$ if it is generated by a strict contact Hamiltonian function. It is readily verified that a contact isotopy is strict if, and only if, it commutes with the Reeb flow (of $\alpha$). 
\item Given smooth functions $g, h: M \to \R$, the contact Poisson bracket is defined by 
\[ \{g, h\}_{\alpha}: = - dg(X_h) - dh(R_{\alpha})\cdot h \]
where $X_h$ is the contact Hamiltonian vector field\footnote{Here, we use the following sign convention: for a Liouville domain $(W, \lambda)$, its symplectic structure is given by $\omega = d\lambda$ and $\iota_{X_H} \omega = \omega(X_H, \cdot) = dH$. By this sign convention, a contact Hamiltonian vector field $X_h$ (generated by $h$) is uniquely determined by the equations $\alpha(X_h)  = -h$ and $\iota_{X_h}d\alpha = dh - dh(R_{\alpha})\cdot\alpha$.}  generated by $h$. Note that $h$ is strict with respect to $\alpha$ if, and only if, $\{h, 1\}_{\alpha} =0$.
\item  A smooth map $F=(f_1,\ldots, f_N):M\to\R^N$ is called contact $\alpha$-involutive if the Poisson brackets $\{f_j, 1\}_\alpha$ and $\{f_i,f_j\}_\alpha$ vanish for all $i,j=1,\ldots, N.$ For brevity, $F: M \to \R^N$ is called contact involutive if it is contact $\alpha$-involutive for some contact form $\alpha$. This notion agrees with Definition 1.7 in \cite{SUV25}. 
\end{enumerate}

\subsection{Contact spectral invariant} 

In this section, we recall the construction of the contact spectral invariant from \cite{DUZ23} that is associated to the unit in symplectic homology. This spectral invariant is extracted from contact Hamiltonian Floer homology ${\rm HF}_\ast(h)$, the homology introduced in \cite{MU19} and further developed in \cite{DU22,UZ22,C23,DUZ23,CHK23,CU24}. The contact Hamiltonian Floer homology ${\rm HF}_\ast(h)$ is associated to a Liouville domain $(W,\lambda)$ and an admissible contact Hamiltonian on its boundary $h_t:\partial W\to\R$. Recall that admissible means $h$ is 1-periodic in time and the corresponding homogeneous Hamiltonian on the symplectization has no 1-periodic orbits. The following are important features of ${\rm HF}_\ast(h).$

\begin{enumerate}[(A)]
	\item There is a well defined continuation map $ {\rm HF}_\ast(h)\to {\rm HF}_\ast(g) $ whenever $h\leqslant g$ are admissible contact Hamiltonians.
	\item The groups ${\rm HF}_\ast(h)$ together with the continuation maps form a directed system of groups indexed by the set of all admissible contact Hamiltonians.  The direct limit of this directed system is equal to Viterbo's symplectic homology ${\rm SH}_\ast(W)$.
\end{enumerate}

Denote by $e$ the unit in ${\rm SH}_\ast(W)$, and by $\alpha$ the contact form $\left.\lambda\right|_{\partial W}$. The contact spectral invariant $c_\alpha(h)$ is defined as follows:
\begin{equation}\label{eq:spec-inv}
c_\alpha(h):= -\inf\left\{ \eta\in \R\:|\: e\in {\rm im}\left( {\rm HF}_\ast(\eta\# h)\to {\rm SH}_\ast(W) \right)  \right\},
\end{equation}
for $h:[0,1]\times \partial W\to\R$. Although not indicated by the notation, the spectral invariant $c_\alpha$ does depend on the filling $W$. Formally speaking, ${\rm HF}_\ast(\eta\# h)$ is well defined only if $\eta\#h$ is admissible. By Lemma~1.2 in \cite{DUZ23},  $\eta\#h$ is admissible for $\eta$ outside of a discrete subset $\mathcal{S}_h$ of $\R$. More precisely, $\mathcal{S}_h$ consists of all the time shifts\label{shift} of the translated points of $\varphi_h^1.$ In the next theorem, we list some of the properties of the spectral invariant.

\begin{theorem}[\cite{DUZ23}]\label{thm:spec}
Let $(W,\lambda)$ be a Liouville domain with $SH_\ast(W)\not=0$. Then, the spectral invariant $c_\alpha$, defined by \eqref{eq:spec-inv}, satisfies the following properties:
\begin{enumerate}
	\item For $a\in\R$, we have $c_\alpha(a)=a$. 
	\item For all $h$, we have $c_\alpha(h)\in \mathcal{S}_h.$
	\item\label{item:stability}  For $h$ and $g$ strict with respect to $\alpha$, we have\footnote{See \cite[Theorem~2.2]{DUZ23} for a general stability property.}
	\[\min_{[0,1]\times \partial W} (h-g)\leqslant c_\alpha(h)- c_\alpha(g)\leqslant \max_{[0,1]\times\partial W}(h-g). \]
	\item \label{item:triangle} For  $h$ and $g$ strict with respect to $\alpha$, we have
	\[ c_\alpha(h\bullet g)\leqslant c_\alpha(h) + c_\alpha( g). \]
	\item \label{item:descent} If $h$ and $g$ give rise to the same element in $\widetilde{\rm Cont}_0(\partial W)$, then $c_\alpha(h)=c_\alpha(g)$. In other words, $c_\alpha$ is well defined on $\widetilde{\rm Cont}_0(\partial W)$.
\end{enumerate}
\end{theorem}
 The properties \eqref{item:triangle} and \eqref{item:descent} in Theorem~\ref{thm:spec} imply that the triangle inequality \eqref{item:triangle} continues to hold if we replace $\bullet$ by $\#$. The properties \eqref{item:stability} and \eqref{item:triangle} have their versions that hold for $h$ and $g$ not necessarily strict. The next proposition states the triangle inequality for general $h$ and $g$.
 
 \begin{prop}[\cite{DUZ23}]\label{prop:trianglenonstrict}
 For all $h,g\in C^\infty([0,1]\times M)$, the following inequality holds
 \[ c_\alpha(h\bullet g)\leqslant c_\alpha(h) + c_\alpha(g) + 4\cdot \max\{ {\rm osc}_\alpha h, {\rm osc}_\alpha g \}. \]
 Here, $ {\rm osc}_\alpha h:= \max_{s,t}\left( \max_M h_t\circ\varphi_R^s - \min_M h_t \right). $
 \end{prop}
 
 The error term in Proposition~\ref{prop:trianglenonstrict} can be fixed by a simple algebraic trick. Namely, one can consider
 \begin{equation}\label{eq:ctilde}
 \tilde{c}_\alpha(h):= c_\alpha(h)+4 {\rm osc}\: h.
\end{equation}
It is not difficult to check that $\tilde{c}_\alpha$ satisfies the triangle inequality without an error term. However, $\tilde{c}(h)$ no longer descends to $\widetilde{\rm Cont}_0(M).$

The contact spectral invariant $c_\alpha$ satisfies a version of conjugacy invariance, that we prove in Section~\ref{sec-proof}.
\begin{prop}\label{prop:conj}
Let $\varphi:M\to M$ be a  strict contactomorphism such that $\varphi\in{\rm Cont}_0(M)$. Then, $c_\alpha(h\circ\varphi)=c_\alpha(h)$ for all contact Hamiltonians $h_t:M\to\R$.
\end{prop}

In fact, the proposition also holds if we, instead of $\varphi\in{\rm Cont}_0(M)$ assume that $\varphi$ is the ideal restriction\footnote{See Section~\ref{sec:conj} for the definition of the ideal restriction.} of a symplectomorphysm $\varphi: \widehat{W}\to\widehat{W}.$

\section{Contact quasi-states and quasi-measures}

In this section, we introduce the notions of partial contact quasi-state and contact quasi-measure.  

\begin{dfn}\label{dfn-pqs}  A {\bf partial contact quasi-state} with respect to $\alpha$ is a functional $\zeta: C^\infty(M) \to  \R$ which satisfies the following axioms:
	\begin{enumerate}[(1)]
		\item \label{pqm:normalization} $\zeta(1)=1$.
		\item\label{pqm:stability} If $f, g\in C^\infty(M)$ are strict with respect to $\alpha$, then
			\[\min_M(f-g)\leqslant \zeta(f) - \zeta(g) \leqslant \max_M (f-g).\]
		\item\label{pqm:homogeneous} For all $f\in C^\infty(M)$ and $s\in\R^+$, we have $\zeta(sf)=s\zeta(f).$
		\item\label{pqm:conjugacy} If $\varphi\in{\rm Cont}_0(M)$ is strict with respect to $\alpha$, then 
		\[\zeta(h\circ\varphi) = \zeta(h)\]
		for all $h\in C(M).$
		\item\label{pqm:vanishing} If $f$ is strict with respect to $\alpha$ and ${\rm supp}\:f$ is displaceable by an element of ${\rm Cont}_0(M),$ then $\zeta(f)=0.$
		\item\label{pqm:triangle} The inequality
		\[ \zeta(f+g)\leqslant \zeta(f) + \zeta(g) \]
		holds for all $f,g\in C^\infty(M)$ such that $0 = \{g,h\}_\alpha = \{g, 1\}_\alpha = \{h, 1\}_\alpha.$
	\end{enumerate}
\end{dfn}

\begin{theorem} \label{thm-qs} Let $(M,\xi)$ be a closed contact manifold that is fillable by a Liouville domain with non-zero symplectic homology. Then, for every contact form $\alpha$ of $\xi$, there exists a contact quasi-state with respect to $\alpha$. \end{theorem}

In Section \ref{sec-proof}, we use a spectral invariant derived from contact Hamiltonian Floer homology (see Section~2 in \cite{DUZ23}) to construct a partial contact quasi-state $\zeta_{\alpha}$ for every contact form $\alpha$ on $(M, \xi)$ as in Theorem \ref{thm-qs}. 

\begin{dfn} \label{dfn-qm} Let $(M, \xi)$ be a closed contact manifold with a contact form $\alpha$. A (subadditive) {\bf contact quasi-measure} with respect to $\alpha$ is a map $\tau: \mathfrak{F}(M) \to  \R$, defined on the set $\mathfrak{F}(M)$ of all closed subsets\footnote{By a standard procedure \cite{PR14}, one can define $\tau$ on open subsets as well. We do not include this definition in the present paper because it is not used in our applications. } of $M$, which satisfies the following axioms:
	\begin{enumerate}[(1)]
		\item \label{qm-normalization}$\tau(M)=1$.
                 \item \label{qm-monotonicity}If $A \subset B$, then $\tau(A) \leq \tau(B)$.
                \item\label{qm-subadditivity} Let $F:M\to\R^N$ be a contact $\alpha$-involutive map and $X_1, \ldots, X_k\subset \R^N$ closed subsets. Denote $A_j:= F^{-1}(X_j)$. Then,
                \[ \tau(A_1\cup \cdots\cup A_k)\leqslant \tau(A_1)+\cdots+ \tau(A_k). \]
\end{enumerate}
\end{dfn}

\begin{cor}\label{cor:qm} Let $(M,\xi)$ be a closed contact manifold that is fillable by a Liouville domain with non-zero symplectic homology. Then, for every contact form $\alpha$ of $\xi$, there exists a (subadditive) contact quasi-measure $\tau$ with respect to $\alpha$. \end{cor}

In fact, we will show that any contact quasi-state induces a (subadditive) quasi-measure. In particular, for any contact form $\alpha$ on $(M, \xi)$, we obtain a quasi-measure from $\zeta_{\alpha}$ mentioned above, which is denoted by $\tau_{\alpha}$. By the vanishing property of $\zeta_{\alpha}$, we have $\tau_{\alpha}(A) = 0$ for any $A$ that is contact displaceable and Reeb-invariant with respect to $\alpha$. This serves as the key step in the proof of Theorem \ref{thm-bft}. 

\section{Proofs} \label{sec-proof}

\subsection{Proof of Proposition~\ref{prop:conj}}\label{sec:conj}

For a Liouville domain $(W,\lambda)$ with boundary $M$, denote by ${\rm Symp}^*(W)$ the group of symplectomorphisms $\psi:\widehat{W}\to\widehat{W}$ of its completion that preserve the Liouville form outside of a compact subset. For each $\psi\in{\rm Symp}^*(W)$ there exists a contactomorphism $\varphi:M\to M$ such that $\psi(x,r) = \left(\varphi(x), \frac{r}{\kappa_\varphi(x)} \right)$ on the conical end for $r$ large enough. Recall that $\kappa_\varphi$ denotes the conformal factor of $\varphi,$ i.e. the positive function $M\to\R^+$ determined by $\varphi^\ast \alpha= \kappa_\varphi\cdot \alpha.$ We call the contactomorphism $\varphi$ \emph{the ideal restriction} of $\psi$ and denote $\varphi:=\psi_\infty$. The key ingredient in the proof of Proposition~\ref{prop:conj} is the existence of the conjugation isomorphisms (cf. \cite[Section~5]{U23})
\[ \mathcal{\psi} : {\rm HF}_\ast(h) \to {\rm HF}_\ast((\psi_\infty)^*h) \]
for all $\psi\in{\rm Symp}^*(W)$ and for admissible contact Hamiltonians $h_t:M\to \R.$ Here, $(\psi_\infty)^*h:= (h\circ \psi_\infty)/{\kappa_{\psi_\infty}}.$ The conjugation isomorphism $\mathcal{\psi}$ commutes with the continuation maps, thus inducing an automorphism of symplectic homology (called and denoted the same). Moreover, the conjugation isomorphism $\mathcal{\psi}:{\rm SH}_\ast(W)\to {\rm SH}_\ast(W)$ preserves the unit.

\begin{proof}[Proof of Proposition~\ref{prop:conj}]
	Let $\varphi:M\to M$ be a strict contactomorphism that is the ideal restriction of some $\psi\in {\rm Symp}^*(W).$ Since $\varphi$ is strict, we have  $\varphi^\ast(\eta\#h)=\eta\#(\varphi^* h)$ for all contact Hamiltonians $h_t:\partial W\to \R$ and all $\eta\in\R.$ Since $\mathcal{\psi}$ commutes with the continuation maps and preserves the unit $e$ in ${\rm SH}_\ast(W)$, we have that
	\[e\in{\rm im}\left( {\rm HF}_\ast(\eta\#h)\to {\rm SH}_\ast(W) \right)\]
	if, and only if,
	\[e\in{\rm im}\left( {\rm HF}_\ast(\eta\#\varphi^*h)\to {\rm SH}_\ast(W) \right).\]
	Consequently, $c_\alpha(h)= c_\alpha(\varphi^*h)=c_\alpha(h\circ\varphi).$
\end{proof}


\subsection{Proof of Theorem~\ref{thm-qs}}

Let $M$ be the boundary of a Liouville domain $(W,\lambda)$ and let $\alpha:=\left.\lambda\right|_{M} $ be the induced contact form on $M$.

\begin{dfn}\label{def:zeta}
	For a smooth function $h\in C^\infty(M)$, define
	\[\zeta_\alpha(h):=\lim_{k\to\infty}\frac{c_\alpha(kh)}{k},\]
	where $c_\alpha$ is as in \eqref{eq:spec-inv} on page~\pageref{eq:spec-inv}.
\end{dfn}

We show here that $\zeta_\alpha$ is a partial contact quasi-state with respect to $\alpha$. We start by proving that $\zeta_\alpha$ is well defined, i.e. that the limit in Definition~\ref{def:zeta} exists and is finite. 

\begin{lemma}
	The sequence $\frac{c_\alpha(kh)}{k}$ converges for all $h\in C^\infty(M).$
\end{lemma}
\begin{proof}
	As in \eqref{eq:ctilde} on page~\pageref{eq:ctilde}, denote $\tilde{c}_\alpha(h):= c_\alpha(h) + 4{\rm osc}\: h.$ The sequence $\tilde{c}_\alpha(kh)$ is subadditive and $\frac{\tilde{c}_\alpha(kh)}{k}$ is bounded from below (by $4 \max h - 3 \min h$). For every such a sequence, the sequence $\frac{\tilde{c}_\alpha(kh)}{k}$ converges to a finite number. 	
	Since $\frac{{\rm osc}\: (kh) }{k}={\rm osc}\: h$ for all $k$, this implies that $\frac{c_\alpha(kh)}{k}$ converges to a finite number as well.
\end{proof}

\begin{prop}\label{prop:yaron}
	Let $U$ be a Reeb invariant subset and let $\varphi_f^t:M\to M$ be a contact isotopy such that $\varphi_f^1(U)\cap U=\emptyset.$ Then, for all $h$ with ${\rm supp}\:h\subset U$, we have\footnote{See page~\pageref{shift} for the definition of $\mathcal{S}_f.$ }
	\[ \mathcal{S}_{h\#f} =\mathcal{S}_f.\]
	In other words, the sets of shifts for translated points of the contactomorphisms $\varphi_h^1\varphi_f^1$ and $\varphi_f^1$ are equal. In fact, the sets of translated points for a given shift coincide for these contactomorphisms.
\end{prop}
\begin{proof}
	It is enough to show that the sets of translated points with shift $\eta$ coincide for the contactomorphisms $\varphi^1_h\varphi^1_f$ and $\varphi_f^1$. That is, it is enough to show ${\rm Fix}(\varphi_R^{-\eta}\varphi^1_h\varphi^1_f)= {\rm Fix}(\varphi_R^{-\eta}\varphi^1_f).$
	Both $\varphi^{-\eta}_R\varphi_h^1\varphi_f^1$ and $\varphi^{-\eta}_R\varphi_f^1$ have no fixed points in $U$. Indeed, if $x\in U$ then $\varphi^1_f(x)\not \in U$ and $\varphi^1_h\varphi_f^1(x)=\varphi_f^1(x)\not \in U.$ Therefore, by the Reeb invariance  of $U$ and its complement, $\varphi_R^{-\eta}\varphi_h^1\varphi_f^1(x)\not\in U$ and $ \varphi_R^{-\eta}\varphi_f^1(x)\not\in U .$ Hence, $x$ cannot be a fixed point of either $\varphi_R^{-\eta}\varphi_h^1\varphi_f^1$ or $\varphi_R^{-\eta}\varphi_f^1$. 
	
	If, on the other hand, $x\not\in U$ then we have the following sequence of equivalent statements:
\begin{align*}
 	x= \varphi_R^{-\eta}\varphi_f^1(x)\quad&\Leftrightarrow\quad \varphi_R^\eta(x)=\varphi_f^1(x)\quad \Leftrightarrow\quad (\varphi_h^1)^{-1}\varphi_R^\eta(x)= \varphi_f^1(x)\\
	&\Leftrightarrow\quad \varphi_R^\eta(x)=\varphi_h^1\varphi_f^1(x) \quad \Leftrightarrow\quad 
	 x= \varphi_R^{-\eta}\varphi_h^1\varphi_f^1(x).
\end{align*}
Thus we complete the proof.
\end{proof}

\begin{cor}\label{cor:vanishing}
	In the situation of Proposition~\ref{prop:yaron}, we have $c_\alpha(h\#f)=c_\alpha(f).$ 
\end{cor}
\begin{proof}
	Proposition~\ref{prop:yaron} implies that $\eta\#(s\cdot h)\#f$, $s\in[0,1]$ is a smooth family of admissible contact Hamiltonians for all $\eta\in\mathcal{S}_f$. Therefore (see Section~7.1 in \cite{DUZ23}), there exists a zig-zag isomorphism ${\rm HF}_\ast(\eta\# f)\to {\rm HF}_\ast(\eta\# h\# f)$ that commutes with the canonical maps ${\rm HF}_\ast(g)\to{\rm SH}_\ast(W).$ Since the zig-zag isomorphisms preserve the unit $e\in {\rm SH}_\ast(W)$, this implies $c_\alpha(h\#f)= c_\alpha(f).$
\end{proof}

\begin{proof}[Proof of Theorem~\ref{thm-qs}]
	We show that $\zeta_\alpha$ is a partial contact quasi-state with respect to $\alpha$. 
	
	The properties \eqref{pqm:normalization}, \eqref{pqm:stability}, and \eqref{pqm:conjugacy} follow directly from the properties of the spectral invariant $c_\alpha$ in Theorem~\ref{thm:spec} and Proposition~\ref{prop:conj}. The property \eqref{pqm:homogeneous} is an immediate consequence of the definition of $\zeta_\alpha$ via the limit. The property \eqref{pqm:triangle} follows from the triangle inequality and the descent property in Theorem~\ref{thm:spec} and the equation $f\#g=f+g$ that holds if $\{f,g\}_\alpha=\{f,1\}_\alpha= \{g,1\}_\alpha=0.$ 
	
	Let us now prove the vanishing property $\eqref{pqm:vanishing}.$ Let $h$ be a strict contact Hamiltonian such that there exists a contact isotopy $\varphi_f^t:M\to M$ satisfying $\varphi_f^1({\rm supp}\:h)\cap {\rm supp}\:h=\emptyset.$ Since $h$ is strict, the set ${\rm supp}\:h$ is Reeb invariant. The triangle inequality from Proposition~\ref{prop:trianglenonstrict} implies
	\[ c_\alpha(h\#f) + c_\alpha(\bar{f})+ 4\max\{ {\rm osc}_\alpha(h\#f) , {\rm osc}_\alpha\bar{f}\}\geqslant c_\alpha(h), \]
	where $\bar{f}$ denotes the contact Hamiltonian of the contact isotopy $(\varphi^{t}_f)^{-1}.$
	Since $h$ is strict, the expression $4\max\{ {\rm osc}_\alpha(h\#f) , {\rm osc}_\alpha\bar{f}\}$ does not depend on $h$. Denote 
	\[K(f):=4\max\{ {\rm osc}_\alpha(h\#f) , {\rm osc}_\alpha\bar{f}\}.\]
	Then Corollary~\ref{cor:vanishing} implies
	\[ c_\alpha(f) + c_\alpha(\bar{f}) + K(f)\geqslant c_\alpha(h). \]
	Since ${\rm supp}\:(kh)= {\rm supp}\:h$ for all positive $k$, we can repeat the same argument for $kh$ instead of $h$ and get
	\[ c_\alpha(f) + c_\alpha(\bar{f}) + K(f)\geqslant c_\alpha(kh). \]
	Hence, $c_\alpha(kh)$ is bounded and $\zeta_\alpha(h)=0.$
\end{proof}

\subsection{Proof of Corollary~\ref{cor:qm}}

\begin{lemma}\label{lem:qsqm}
	Let $\zeta$ be a partial contact quasi-state on $M$ with respect to $\alpha$. Then, $\tau: \mathfrak{F}(M)\to \R$ defined by
	\[ \tau(A):= \inf\left\{ \zeta(h)\:|\: h:M\to[0,1] \text{ strict, smooth and } \left.h\right|_A=1 \right\} \]
	is a subadditive contact quasi-measure with respect to $\alpha$.
\end{lemma}
\begin{proof}
The properties \eqref{qm-normalization} and \eqref{qm-monotonicity} in Definition~\ref{dfn-qm} are easily verified for $\tau.$ Let us show that $\tau$ satisfies property \eqref{qm-subadditivity}, that is that $\tau$ is subadditive. It is enough to prove the subadditivity for two subsets. Let $F:M\to\R^N$ be a contact $\alpha$-involutive map and let $X, Y\subset \R^N$ be two closed subsets. Denote $A:=F^{-1}(X)$ and $B:=F^{-1}(Y).$ Without loss of generality assume $X, Y\subset F(M).$
The stability property \eqref{pqm:stability} in Definition~\ref{dfn-pqs} enables us to extend $\zeta$ to the set of continuous functions $h:M\to \R$ that are strict (i.e. Reeb-invariant). With this extension, we have
\[\tau(A)= \inf\{ \zeta(h)\:|\: h:M\to[0,1]\text{ strict, continuous, and }\left. h \right|_{A}=1\}.\]
	Now, we show
\[ \tau(A)=\inf\{ \zeta(\mu\circ F)\:|\: \mu:F(M)\to[0,1]\text{ continuous and} \left.\mu\right|_X=1\}. \]
It is enough to show that for each strict, continuous $h$ with $\left.h\right|_{A}=1$ there exists continuous $\mu:\R^N\to[0,1]$ satisfying $\mu\circ F\leqslant h$ and $\left.\mu\right|_{X}=1.$ One can take $\mu$ to be the function defined by
\[ \mu(x):=\inf \{h(y)\:|\: y\in F^{-1}(x)\}.\]
In other words, $\mu$ is obtained from $h$ by taking the infima over the fibres of $F$. 

Let $\mu_a, \mu_b:F(M)\to[0,1]$ be continuous functions such that $\left.\mu_a\right|_A=1$ and $\left.\mu_b\right|_B=1.$ Since the functions $\mu_a\circ F$ and $\mu_b\circ F$ are strict with respect to $\alpha$ and since $F$ is involutive, by subadditivity of $\zeta$, we get
\[\tau(A\cup B)\leqslant \zeta( h )\leqslant \zeta(\mu_a\circ F + \mu_b\circ F) \leqslant \zeta(\mu_a\circ F) + \zeta(\mu_b\circ F), \]
where $h:=\min\{1, \mu_a\circ F + \mu_b\circ F\}$. Taking infimum over $\mu_a$ and $\mu_b$ yields $\tau(A\cup B)\leqslant \tau(A)+\tau(B).$
\end{proof}

\begin{proof}[Proof of Corollary~\ref{cor:qm}]
	The proof follows from Lemma~\ref{lem:qsqm} and Theorem~\ref{thm-qs}.
\end{proof}
	
Now, let us show that the contact quasi-measure in Lemma~\ref{lem:qsqm} satisfies the vanishing property.

\begin{lemma}\label{lem:qm-vanishing}
	Let $\zeta$ be a partial contact quasi-state on $M$ with respect to $\alpha$. Denote by $\tau$ the induced contact quasi-measure. Let $A\subset M$ be the preimage of a closed subset under a contact $\alpha$-involutive map $F:M\to\R^n$. If $A$ is displaceable by ${\rm Cont}_0(M)$, then $\tau(A)=0.$
\end{lemma}
\begin{proof}
Since $A\subset M$ is closed, there exists an open neighbourhood $U\supset A$ that is displaceable by ${\rm Cont}_0(M).$ Let $\chi: M\to [0,1]$ be a smooth Reeb-invariant function with support in $U$ such that $\left.\chi\right|_{A}=1.$ Then,
\[ \tau(A)=\inf\left\{ \zeta(\chi\cdot h)\:|\: h:M\to[0,1]  \text{ strict, smooth,  and }\left.h\right|_{A}=1\right\}. \]
Since ${\rm supp}\:(\chi \cdot h)$ is displaceable by ${\rm Cont}_0(M),$ the vanishing property of $\zeta$ implies $\tau(A)=0.$
\end{proof}

\subsection{Proof of Theorem~\ref{thm-bft}}

Assume the contrary, i.e. that there exists a contact involutive map $F:M\to\R^N$ on the boundary $M=\partial W$ of a Liouville domain $(W,\lambda)$ with ${\rm SH}_\ast(W)\not=0$ whose all the fibres are displaceable by ${\rm Cont}_0(M).$ By definition, $F$ is contact involutive if it is contact $\alpha$-involutive for some contact form $\alpha$ on $M$. By modifying $W$ inside its completion if necessary, we may assume $\alpha=\left.\lambda\right|_{\partial W}$. Since $M$ is compact and the fibres of $F$ are all displaceable, there exist closed sets $A_j= F^{-1}(X_j), j=1,\ldots, k$ that are displaceable by ${\rm Cont}_0(M)$ and whose interiors cover $M$. Now, using \eqref{qm-normalization} and \eqref{qm-subadditivity} in Definition~\ref{dfn-qm}, and Lemma~\ref{lem:qm-vanishing}, we get the following contradiction
\[  1=\tau_\alpha(M) = \tau_\alpha(A_1\cup\cdots\cup A_k)\leqslant \tau_\alpha(A_1)+\cdots+\tau_\alpha(A_k)=0.\]

\medskip

\noindent {\bf Acknowledgement}.  A major part of this work was done while the first author was visiting the Institute of Geometry and Physics at the University of Science and Technology of China in Hefei, and he is grateful for the invitation and hospitality.

\bibliographystyle{amsplain}
\bibliography{biblio_Q_BFT}

\providecommand{\bysame}{\leavevmode\hbox to3em{\hrulefill}\thinspace}
\providecommand{\MR}{\relax\ifhmode\unskip\space\fi MR }
\providecommand{\MRhref}[2]{%
  \href{http://www.ams.org/mathscinet-getitem?mr=#1}{#2}
}
\providecommand{\href}[2]{#2}
\begin{thebibliography}{10}

\bibitem{C23}
Dylan Cant, \emph{Shelukhin's hofer distance and a symplectic cohomology
  barcode for contactomorphisms}, arXiv preprint arXiv:2309.00529 (2023).

\bibitem{C2024}
\bysame, \emph{Remarks on eternal classes in symplectic cohomology}, arXiv
  preprint arXiv:2410.03914 (2024).

\bibitem{CHK23}
Dylan Cant, Jakob Hedicke, and Eric Kilgore, \emph{Extensible positive loops
  and vanishing of symplectic cohomology}, arXiv preprint arXiv:2311.18267
  (2023).

\bibitem{CU24}
Dylan Cant and Igor Uljarevi{\'c}, \emph{Selective floer cohomology for contact
  vector fields}, arXiv preprint arXiv:2405.05443 (2024).

\bibitem{DUZ23}
Danijel Djordjevi{\'c}, Igor Uljarevi{\'c}, and Jun Zhang, \emph{Quantitative
  characterization in contact hamiltonian dynamics--{I}}, arXiv preprint
  arXiv:2309.00527 (2023).

\bibitem{DU22}
Du{\v{s}}an Drobnjak and Igor Uljarevi{\'c}, \emph{Exotic symplectomorphisms
  and contact circle actions}, Communications in Contemporary Mathematics
  \textbf{24} (2022), no.~03, 2150045.

\bibitem{EP06}
Michael Entov and Leonid Polterovich, \emph{Quasi-states and symplectic
  intersections}, Commentarii Mathematici Helvetici \textbf{81} (2006), no.~1,
  75--99.

\bibitem{MU19}
Will~J. Merry and Igor Uljarevi\'{c}, \emph{Maximum principles in symplectic
  homology}, Israel Journal of Mathematics \textbf{229} (2019), 39--65.

\bibitem{PR14}
Leonid Polterovich and Daniel Rosen, \emph{Function theory on symplectic
  manifolds}, vol.~34, American Mathematical Soc., 2014.

\bibitem{SUV25}
Yuhan Sun, Igor Uljarevic, and Umut Varolgunes, \emph{Contact big fiber
  theorems}, arXiv preprint arXiv:2503.04277 (2025).

\bibitem{U23}
Igor Uljarevi{\'c}, \emph{Selective symplectic homology with applications to
  contact non-squeezing}, Compositio Mathematica \textbf{159} (2023), no.~11,
  2458--2482.

\bibitem{UZ22}
Igor Uljarevi{\'c} and Jun Zhang, \emph{Hamiltonian perturbations in contact
  floer homology}, Journal of Fixed Point Theory and Applications \textbf{24}
  (2022), no.~4, 71.

\end{thebibliography}

\noindent\\

\end{document}